\documentclass[12pt]{article}
\usepackage{amsfonts}
\usepackage{bm}
\usepackage{amsmath, amsthm, amsfonts, amssymb,  color}
\usepackage{mathrsfs}
\usepackage{graphicx}
\newtheorem{thm}{Theorem}[section]

\newtheorem{lem}[thm]{Lemma}

\theoremstyle{definition}

\newcommand{\scr}[1]{\mathscr #1}
\definecolor{wco}{rgb}{0.5,0.2,0.3}

\numberwithin{equation}{section} \theoremstyle{remark}
\newtheorem{rem}{Remark}[section]

\newcommand{\ua}{\uparrow}

\renewcommand{\hat}{\widehat}

\title{{\bf  Convergence Rate of Euler-Maruyama Scheme for SDDEs of Neutral Type}
}
\author{ Yanting Ji,\thanks{Department of
		Mathematics, Swansea University, Singleton Park, SA2 8PP, UK,
		mathjyt@gmail.com} \and Jianhai Bao,\thanks{School of Mathematics
		and Statistics, Central South University, China,
		jianhaibao13@gmail.com}  \and Chenggui Yuan \thanks{Department of
Mathematics, Swansea University, Singleton Park, SA2 8PP, UK,
C.Yuan@swansea.ac.uk}}

\date{}

\begin{document}
\def\R{\mathbb R}  \def\ff{\frac} \def\ss{\sqrt} \def\B{\mathbf
B}
\def\N{\mathbb N} \def\kk{\kappa} \def\m{{\bf m}}
\def\dd{\delta} \def\DD{\Dd} \def\vv{\varepsilon} \def\rr{\rho}
\def\<{\langle} \def\>{\rangle} \def\GG{\Gamma}
  \def\nn{\nabla} \def\pp{\partial} \def\EE{\scr E}
\def\d{\text{\rm{d}}} \def\bb{\beta} \def\aa{\alpha} \def\D{\scr D}
  \def\si{\sigma} \def\ess{\text{\rm{ess}}}
\def\beg{\begin} \def\beq{\begin{equation}}  \def\F{\scr F}
\def\Ric{\text{\rm{Ric}}} \def\Hess{\text{\rm{Hess}}}
\def\e{\text{\rm{e}}} \def\ua{\underline a} \def\OO{\Omega}  \def\oo{\omega}
 \def\tt{\tilde} \def\Ric{\text{\rm{Ric}}}
\def\cut{\text{\rm{cut}}} \def\P{\mathbb P} \def\ifn{I_n(f^{\bigotimes n})}
\def\C{\scr C}      \def\aaa{\mathbf{r}}     \def\r{r}
\def\gap{\text{\rm{gap}}} \def\prr{\pi_{{\bf m},\varrho}}  \def\r{\mathbf r}
\def\Z{\mathbb Z} \def\vrr{\varrho} \def\l{\lambda}
\def\L{\scr L}\def\Tt{\tt} \def\TT{\tt}\def\II{\mathbb I}
\def\i{{\rm in}}\def\Sect{{\rm Sect}}\def\E{\mathbb E} \def\H{\mathbb H}
\def\M{\scr M}\def\Q{\mathbb Q} \def\texto{\text{o}} \def\LL{\Lambda}
\def\Rank{{\rm Rank}} \def\B{\scr B} \def\i{{\rm i}} \def\HR{\hat{\R}^d}
\def\to{\rightarrow}\def\l{\ell}\def\lf{\lfloor}\def\rf{\rfloor}
\def\8{\infty}\def\ee{\epsilon} \def\Y{\mathbb{Y}} \def\lf{\lfloor}
\def\rf{\rfloor}\def\3{\triangle} \def\O{\mathcal {O}}
\def\SS{\mathbb{S}}\def\ta{\theta}\def\h{\hat}
\def\1{\lesssim}
\def\la{\langle}\def\ra{\rangle}

\def\trace{\text{\rm{trace}}}
\renewcommand{\bar}{\overline}
\def\Y{\mathbb{Y}}
\renewcommand{\tilde}{\widetilde}

\maketitle

\begin{abstract}
In this paper, we are   concerned with  convergence rate of
Euler-Maruyama (EM)
 scheme for stochastic differential delay equations (SDDEs) of {\it neutral type},
 where the neutral term, the drift term and the diffusion term are allowed
 to be of polynomial growth. More precisely, for SDDEs of neutral
 type driven by Brownian motions, we reveal that the convergence
 rate of the corresponding EM scheme is one half; Whereas  for SDDEs of neutral
 type driven by jump processes, we show that the best  convergence
 rate of the associated EM scheme is close to one half.

\smallskip

    \smallskip

\end{abstract}

\section{Introduction}
There is numerous literature concerned with convergence rate of
numerical schemes for stochastic differential equations (SDEs).
Under a H\"older condition, Gy\"ongy-R\'asonyi \cite{Gyo} provided a
convergence rate of EM scheme; Under the Khasminskii-type condition,
Mao \cite{Mao16} revealed that the convergence rate of the truncated
EM method is close to one half; Sabanis \cite{S13} recovered the
classical rate of convergence (i.e., one half)  for the tamed Euler
schemes, where, for the SDE involved,  the drift coefficient
satisfies a one-side Lipschitz condition and a polynomial Lipschitz
condition, and the diffusion term is Lipschitzian. There is also
some literature on convergence rate of numerical schemes for
stochastic functional differential equations (SFDEs). For example,
under a log-Lipschitz condition, Bao et al. \cite{Bao11} studied
convergence rate of EM approximation for a range of SFDEs driven by
jump processes; Bao-Yuan \cite{Bao} investigated convergence rate of
EM approach for a class of SDDEs, where the drift and diffusion
coefficients are allowed to be polynomial growth with respect to the delay
variables; Gy\"ongy-Sabanis \cite{GS13} discussed   rate of almost
sure convergence of Euler approximations for SDDEs under
monotonicity conditions.

Increasingly real-world systems are modelled by SFDEs of neutral
type,
 as they
represent systems which evolve in a random environment and whose
evolution depends on the past states and derivatives of states of
the systems through either memory or time delay. In the last decade,
for  SFDEs of neutral type, there are a large number of papers on,
e.g.,  stochastic stability (see, e.g., \cite{BHY,Mao,MSY}), on
large fluctuations (see, e.g., \cite{AAM}), on large deviation
principle (see, e.g., \cite{BY}), on transportation inequality (see,
e.g., \cite{BWY}), to name a few.

So far, the topic on numerical approximations for SFDEs of neutral
type has also been investigated considerably. For instance, under a
global Lipschitz condition, Wu-Mao \cite{Wu} revealed convergence
rate of the   EM scheme constructed is close to one half; under a
log-Lipschitz condition, Jiang et al. \cite{Jiang} generalized
\cite{YM} by Yuan and Mao to the neutral case; under the
Khasminskill-type condition, following the line of Yuan-Glover
\cite{YG}, \cite{M11,Zhou1} studied convergence in probability of
the associated EM scheme; for preserving stochastic stability (of
the exact solutions) of variable numerical schemes, we refer to,
e.g., \cite{LC,Yu,Zhou,Zong} and references therein.

We remark that most of the existing literature on convergence rate
of EM scheme for SFDEs of neutral type is dealt with under the
global Lipschitz condition, where, in particular,   the neutral term
is contractive. Consider the following SDDE of neutral type
\begin{equation}\label{example}
\d\{X(t)-X^{2}(t-\tau)\}  = \{aX(t)+bX^{3}(t-\tau)\}\d
t+cX^{2}(t-\tau)\d B(t),~t\ge0,
\end{equation}
in which $a,b,c \in\mathbb{R}$, $\tau >0$ is some constant, and
$B(t)$ is  a scalar Brownian motion. Observe that all the neutral,
drift and diffusion coefficients in \eqref{example} are highly
linear with respect to the delay variable so that the existing
results on convergence rate of EM schemes associated with SFDEs of
neutral type cannot be applied to the example above. So in this paper we
intend to establish the  theory on convergence rate of EM scheme for
a class of SDDEs of neutral type, where, in particular, the neutral
term is of polynomial growth, so that it covers more interesting
models.

Throughout the paper, the shorthand notation $a\lesssim b$ is used to express that there exists a positive constant $c$ such that $a\le cb,$ where $c$ is a generic constant whose value may change from line to line. Let $(\Omega,\mathcal{F},\P)$ be a complete probability space with a filtration $(\mathcal{F}_t)_{t\geq 0}$
satisfying the usual conditions (i.e., it is right continuous and
$\mathcal{F}_{0}$ contains all $\P$-null sets). For each integer
$n\ge1$, let $(\R^n, \<\cdot,\cdot\>, |\cdot|)$ be an
$n$-dimensional Euclidean space. For $A\in\R^n\otimes\R^m$, the
collection of all $n\times m$ matrices, $\|A\|$ stands for the
Hilbert-Schmidt norm, i.e., $\|A\|=(\sum_{i=1}^m|Ae_i|^2)^{1/2}$,
where $(e_i)_{i\ge1}$ is the orthogonal basis of $\R^m.$ For
$\tau>0$, which is referred to as delay or memory,
$\C:=C([-\tau,0];\mathbb{R}^{n})$ means the space of all continuous
functions $\phi:[-\tau,0]\mapsto\mathbb{R}^{n}$ with the uniform
norm $\|\phi\|_{\infty}:=\sup_{-\tau\leq\theta\leq 0}|\phi(\theta)|$. Let
$(B(t))_{t\ge0}$ be a standard $m$-dimensional Brownian motion
defined on the probability space
$(\Omega,\mathcal{F},(\mathcal{F}_t)_{t\ge0},\P)$.

To begin,  we focus on an SDDE  of neutral type on $(\R^n,
\<\cdot,\cdot\>, |\cdot|)$ in the form
\begin{equation}\label{NSDDE}
\d\{X(t)-G(X(t-\tau))\} = b(X(t),X(t-\tau))\d
t+\sigma(X(t),X(t-\tau))\d B(t),~t>0
\end{equation}
with the initial value $X(\theta)=\xi(\theta)$ for
$\theta\in[-\tau,0],$ where $G:\R^n\mapsto\R^n$,
$b:\R^n\times\R^n\mapsto\R^n$, $\si:\R^n\times\R^n\mapsto\R^{n\times
m}$.

We assume that there exist constants $L>0$ and $q\ge1$ such that,
for $x,y,\bar{x},\bar{y}\in \mathbb{R}^{n}$,
\begin{description}

\item[(A1)]$ |G(y)-G(\bar{y})|\leq L(1+|y|^q+|\bar y|^q)|y-\bar{y}|;$

\item[(A2)]$
|b(x,y)-b(\bar{x},\bar{y})|+\|\sigma(x,y)-\sigma(\bar{x},\bar{y})\|
\leq L|x-\bar{x}|+L(1+|y|^q+|\bar y|^q)|y-\bar{y}|$, where
$\|\cdot\|$ stands for the Hilbert-Schmidt norm;
\item[(A3)] $|\xi(t)-\xi(s)|\leq L|t-s|$ for any $s,t\in[-\tau,0]$.
\end{description}

\begin{rem}
{\rm There are some examples such that {\bf(A1)} and {\bf(A2)} hold.
For instance, if $G(y)=y^2, b(x,y)=\si(x,y)=ax+y^3$ for any
$x,y\in\R$ and some $a\in\R,$ Then  both  {\bf(A1)} and {\bf(A2)}
hold  by taking $V(y,\bar y)=1+\ff{3}{2}y^2+\ff{3}{2}\bar y^2$ for
arbitrary $ y, \bar y\in\R$.}
\end{rem}

By following a similar argument to  \cite[Theorem 3.1, p.210]{Mao},
\eqref{NSDDE} has a unique solution $\{X(t)\}$ under ({\bf A1}) and
({\bf A2}). In the sequel,
 we introduce the EM scheme associated with (\ref{NSDDE}). Without loss of generality, we assume that $h=T/M=\tau/m\in(0,1)$
for some integers $M,m>1$. For every integer $k = -m,\cdots,0,$ set
$Y_{h}^{(k)}:=\xi(kh),$ and for each integer $k = 1, \cdots, M-1,$
we define
\begin{equation}\label{DEulerschemeStep}
Y_{h}^{(k+1)}-G(Y_{h}^{(k+1-m)})=
Y_{h}^{(k)}-G(Y_{h}^{(k-m)})+b(Y_{h}^{(k)},Y_{h}^{(k-m)})h+\sigma(Y_{h}^{(k)},Y_{h}^{(k-m)})\Delta
B^{(k)}_{h},
\end{equation}
where $\Delta B^{(k)}_{h}:= B((k+1)h)-B(kh).$ For any
$t\in[kh,(k+1)h),$ set $\bar{Y}(t): = Y_{h}^{(k)}.$ To avoid the
complex calculation, we define the continuous-time EM approximation
 solution $Y(t)$ as below: For any
$\theta\in[-\tau,0],$ $Y(\theta) = \xi(\theta), $ and
\begin{equation}\label{CEulerscheme}
\begin{split}
Y(t)=& G(\bar{Y}(t-\tau))+\xi(0)-G(\xi(-\tau))+\int_{0}^{t}b(\bar{Y}(s),\bar{Y}(s-\tau))\d s\\
&+\int_{0}^{t}\sigma(\bar{Y}(s),\bar{Y}(s-\tau))\d B(s),
~~~~~t\in[0,T].
\end{split}
\end{equation}
A straightforward calculation shows that the continuous-time EM
approximate solution $Y(t)$ coincides with the discrete-time
approximation solution $\bar{Y}(t)$ at the grid points $t = nh.$

\smallskip

 The first main result in this paper is stated as below.
\begin{thm}\label{ConvergenceRate}
    {\rm Under the  assumptions {\bf (A1)}-{\bf (A3)},
        \begin{equation}\label{0}
        \E\bigg(\sup_{0\le t\le
        T}|X(t)-Y(t)|^{p}\bigg)\1h^{p/2},~~~p\geq2.
        \end{equation}
So the convergence rate of the EM scheme (i.e.,
\eqref{CEulerscheme}) associated with (\ref{NSDDE}) is one half.
 }
\end{thm}

Next, we move to consider the convergence rate of EM scheme
corresponding to a class of SDDEs of neutral type driven by pure
jump processes. More precisely,  we consider an SDDEs of neutral
type
\begin{equation}\label{NSDDEJump}
\d\{X(t)-G(X(t-\tau))\}=b(X(t),X(t-\tau))\d
t+\int_{U}g(X(t-),X((t-\tau)-),u)\tilde{N}(\d u,\d t)
\end{equation}
with the initial data $X(\theta)=\xi(\theta), \theta\in[-\tau,0]$.
Herein,  $G$ and $b$ are given as in \eqref{NSDDE},
$g:\R^n\times\R^n\times U\mapsto\R^m$, where $U\in\B(\R)$;
$\tilde{N}(\d t,\d u):= N(\d t,\d u)-\d t\lambda(\d u)$ is the
compensated Poisson measure associated with the Poisson counting
measure $N(\d t,\d u)$ generated by a
 stationary $\mathcal{F}_{t}$-Poisson point process $\{p(t)\}_{t\ge0}$ on $\R$ with characteristic measure
 $\lambda(\cdot),$ i.e., $N(t,U)= \sum_{s\in D(P),s\leq t}I_{U}(p(s))$ for $U\in \mathcal{B}(\R);$
$X(t-):=\lim_{s \uparrow t} X(s).$

We assume that $b$ and $G$ such that ({\bf A1}) and ({\bf A2}) with
$\si\equiv{\bf0}_{n\times m}$ therein. We further suppose that there
exist $L_0,r>0$ such that for any $x,y,\bar{x},\bar{y}\in
\mathbb{R}^{n}$ and $u\in U,$
\begin{description}
\item[(A4)]
$|g(x,y,u)-g(\bar{x},\bar{y},u)|\leq L_0 (|x-\bar{x}|+(1+|y|^q+|\bar
y|^q)|y-\bar{y}|)|u|^r $ and $|g(0,0,u)|\leq |u|^r,$ where $q\ge1$
is the same as that in ({\bf A1}).

\item[(A5)] $\int_{U}|u|^p\lambda(du) <\infty$ for any $p\geq 2.$

\end{description}

\begin{rem}
{\rm The jump coefficient may also be highly non-linear with respect
to the delay argument; for example, $x,y\in\R,$ $u\in U$ and
$q\geq1,$ $g(x,y,u)=(x+y^{q})u$ satisfies ({\bf A5}).}
\end{rem}

By following the procedures of \eqref{DEulerschemeStep} and
\eqref{CEulerscheme}, the discrete-time EM scheme and the
continuous-time   EM approximation associated with \eqref{NSDDEJump}
are defined respectively as below:
\begin{equation}\label{DEulerschemeStepJump}
Y_{h}^{(n+1)}-G(Y_{h}^{(n+1-m)})=
Y_{h}^{(n)}-G(Y_{h}^{(n-m)})+b(Y_{h}^{(n)},Y_{h}^{(n-m)})h+g(Y_{h}^{(n)},Y_{h}^{(n-m)},u)\Delta
\tilde{N}_{nh},
\end{equation}
where $\Delta \tilde{N}_{nh}:= \tilde{N}((n+1)h,U)-\tilde{N}(nh,U),$
and
\begin{equation}\label{CEulerschemeJump}
\begin{split}
Y(t)=& G(\bar{Y}(t-\tau))+\xi(0)-G(\xi(-\tau))+\int_{0}^{t}b(\bar{Y}(s),\bar{Y}(s-\tau))\d s\\
&+\int_{0}^{t}\int_{U}g(\bar{Y}(s-),\bar{Y}((s-\tau)-),u)\tilde{N}(\d
u,\d s),
\end{split}
\end{equation}
where $\bar Y$ is defined similarly as in \eqref{CEulerscheme}.

\smallskip

Our second main result in this paper is presented as follows.
\begin{thm}\label{ConvergenceRateJump}
{\rm Under({\bf A1})-({\bf A5}) with $\si\equiv{\bf0}_{n\times m}$
therein, for any $p\geq2$ and $\theta\in(0,1)$
\begin{equation}\label{22}
\E\Big(\sup_{0\leq t\leq T}|X(t)-Y(t)|^{p}\Big)\1
h^{\frac{1}{(1+\theta)^{[T/\tau]}}}.
\end{equation}
So the best  convergence rate of EM scheme (i.e.,
\eqref{CEulerschemeJump}) associated with \eqref{NSDDEJump} is close
to one half.

}
\end{thm}

\begin{rem}
{\rm By a close inspection of the proof for Theorem
\ref{ConvergenceRateJump}, the conditions ({\bf A4}) and ({\bf A5})
can be replaced by: For any $p>2$ there exists $K_p,K_0>0$ and $q>1$
such that
\begin{equation*}
\begin{split}
&\int_U|g(x,y,u)|^p\lambda(\d u)\le K_p(1+|x|^p+|y|^q);\\
&\int_U|g(x,y,u)-g(\bar x,\bar y,u)|^p\lambda(\d u)\le K_p[|x-\bar
x|^p+(1+|y|^q+|\bar y|^q)|y-\bar y|^p];\\
&\int_U|g(x,y,u)|^2\lambda(\d u)\le K_0(1+|x|^2+|y|^q);\\
&\int_U|g(x,y,u)-g(\bar x,\bar y,u)|^2\lambda(\d u)\le K_0[|x-\bar
x|^2+(1+|y|^q+|\bar y|^q)|y-\bar y|^2]
\end{split}
\end{equation*}
for any $x,y,\bar x,\bar y\in\R^n.$

}
\end{rem}

\section{Proof of Theorem \ref{ConvergenceRate}}

The lemma below  provides estimates of  the $p$-th moment of the
 solution to \eqref{NSDDE} and the corresponding EM scheme, alongside with the $p$-th moment
of the displacement.

\begin{lem}\label{L0}
{\rm Under ({\bf A1}) and ({\bf A2}),  for any $p\ge 2$ there exists
a constant $C_T>0$ such that
\begin{equation}\label{00}
\E\Big(\sup_{0\le t\le T}|X(t)|^p\Big)\vee\E\Big(\sup_{0\leq t\leq T}|Y(t)|^{p}\Big)\leq C_T,
\end{equation}
and
\begin{equation}\label{eq4}
\E\Big(\sup_{0\le t\le T}|\GG(t)|^{p}\Big)\1h^{p/2},
\end{equation}
where $\GG(t):=Y(t)-\bar{Y}(t)$.
}
\end{lem}
\begin{proof}
We  focus only on the following estimate
\begin{equation}\label{bao}
\E\Big(\sup_{0\leq t\leq T}|Y(t)|^{p}\Big)\leq C_T
\end{equation}
for some constant $C_T>0$  since the uniform $p$-th moment of $X(t)$
in a finite time interval can be done similarly. From ({\bf A1}) and
({\bf A2}), one has
\begin{equation}\label{02}
|G(y)|\1 1+|y|^{1+q},
\end{equation}
and
\begin{equation}\label{03}
|b(x,y)|+\|\sigma(x,y)\|\1 1+|x|+|y|^{1+q}
\end{equation}
for any  $x,y\in\R^n.$ By the H\"older inequality, the
Burkhold-Davis-Gundy (B-D-G) inequality (see, e.g., \cite[Theorem
7.3, p.40]{Mao}),  we  derive from  \eqref{02} and \eqref{03}
 that
\begin{equation*}
\begin{split}
\E\Big(\sup_{-\tau\le s\le t}|Y(s)|^p\Big)
&\1 1+\|\xi\|_\8^{p(1+q)}+\E\Big(\sup_{-\tau\le s\le t-\tau}|\bar{Y}(s)|^{p(1+q)}\Big)\\
&\quad+\int_0^t\{\E|\bar{Y}(s)|^p+\E|\bar{Y}(s-\tau)|^{p(1+q)}\}\d s\\
&\11+\|\xi\|_\8^{p(1+q)}+\E\Big(\sup_{-\tau\le s\le
t-\tau}|Y(s)|^{p(1+q)}\Big)\\
&\quad+\int_0^T\E\Big(\sup_{-\tau\le r\le s}|Y(r)|^p\Big)\d s,
\end{split}
\end{equation*}
where we have used $Y(kh)=\bar Y(k h)$ in the last display. This,
together with Gronwall's inequality, yields that
\begin{equation*}
\E\Big(\sup_{0\le s\le t}|Y(s)|^p\Big)\1
 1+  \|\xi\|^{p(1+q)}_\8+\E\Big(\sup_{0\le s\le
 (t-\tau)\vee0}|X(s)|^{p(1+q)}\Big),
\end{equation*}
which further implies that
\begin{equation*}
\E\Big(\sup_{0\le t\le \tau}|X(t)|^{p}\Big)\11+ \|\xi\|^{p(1+q)}_\8,
\end{equation*}
and that
\begin{equation*}
\begin{split}
\E\Big(\sup_{0\le t\le 2\tau}|X(t)|^p\Big)&\1 1+
\E\|\xi\|^{p(1+q)}_\8+\Big(\sup_{0\le t\le
\tau}|X(t)|^{p(1+q)}\Big)\11+ \|\xi\|^{p(1+q)^{2}}_\8.
\end{split}
\end{equation*}
Thus \eqref{bao} follows from an inductive argument.

\smallskip

 Employing H\"older's inequality and BDG's inequality, we deduce
from \eqref{CEulerscheme} and \eqref{03} that
\begin{equation*}
\begin{split}
\E\Big(\sup_{0\le t\le T}|\GG(t)|^{p}\Big)
&\1\sup_{0\le k\le M-1}\Big\{\E\Big(\sup_{kh\le t \leq (k+1)h}\Big|\int_{kh}^tb(\bar{Y}(s),\bar{Y}(s-\tau))\d s\Big|^p\Big)\\
&\quad+\E\Big(\sup_{nh\le t
\leq(k+1)h}\Big|\int_{kh}^t\sigma(\bar{Y}(s),\bar{Y}(s-\tau))\d
B(s)\Big|^p\Big)\Big\}\\
&\1\sup_{0\le k\le M-1}\Big\{h^{p-1}\E\int_{kh}^{(k+1)h}|b(\bar{Y}(s),\bar{Y}(s-\tau))|^p\d s\\
&\quad+h^{\ff{p}{2}-1}\E\int_{kh}^{(k+1)h}\|\si(\bar{Y}(s),\bar{Y}(s-\tau))\|^p\d s\Big\}\\
&\1 h^{\ff{p}{2}-1}\sup_{0\le k\le M-1}\Big\{\int_{kh}^{(k+1)h}\big(1+\E|\bar{Y}(s)|^{p}+\E|\bar{Y}(s-\tau)|^{p(q+1)}\big)\d s \Big\}\\
&\1h^{\ff{p}{2}}.
\end{split}
\end{equation*}
where in the last step we have used \eqref{bao}. The desired
assertion is therefore complete.
\end{proof}

With Lemma \ref{L0} in hand, we are now in the position to finish
the proof of Theorem \ref{ConvergenceRate}.

\noindent{\bf Proof of Theorem \ref{ConvergenceRate}.} We follow the
Yamada-Watanabe approach (see, e.g., \cite{Bao}) to complete the
proof of Theorem \ref{ConvergenceRate}. For fixed $\kappa >1$ and
arbitrary $\vv \in(0,1),$ there exists a continuous non-negative
function $\varphi_{\kappa\vv}(\cdot)$  with the support
$[\vv/\kappa,\vv]$ such that
\begin{equation*}
\int_{\vv/\kappa}^\vv\varphi_{\kappa\vv}(x)\d x = 1 \quad \mbox{and}
\quad \varphi_{\kappa\vv}(x) \leq \frac{2}{x\ln\kappa}, \quad x> 0.
\end{equation*}
Set
\begin{equation*}
\phi_{\kappa\vv}(x):=
\int_{0}^{x}\int_{0}^{y}\varphi_{\kappa\vv}(z)\d z\d y, \quad x>0.
\end{equation*}
It is readily to see that $\phi_{\kappa\vv}(\cdot)$ such that
\begin{equation}\label{FunctionProperty}
x-\vv\leq \phi_{\kappa\vv}(x)\leq x, \quad x > 0.
\end{equation}
Let
\begin{equation}\label{b5}
V_{\kappa\vv}(x)= \phi_{\kappa\vv}(|x|), \quad x\in\mathbb{R}^{n}.
\end{equation}
By a straightforward calculation, it holds
\begin{equation*}
(\nabla V_{\kappa\vv})(x)=
\phi^{'}_{\kappa\vv}(|x|)|x|^{-1}x,~~~~x\in\mathbb{R}^{n}
\end{equation*}
and
\begin{equation*}
(\nabla^2
V_{\kappa\vv})(x)=\phi^{'}_{\kappa\vv}(|x|)(|x|^{2}\mbox{\bf
I}-x\otimes x)|x|^{-3}+|x|^{-2}\phi^{''}_{\kappa\vv}(|x|)x\otimes
x,~~~~x\in\mathbb{R}^{n},
\end{equation*}
where $\nabla$ and $\nabla^2$ stand for the gradient and Hessian
operators, respectively, $\mbox{\bf I}$ denotes the identical
matrix, and $x\otimes x=xx^*$ with $x^*$ being the transpose of
$x\in\R^n.$ Moreover, we have
\begin{equation}\label{FunctionBound}
|(\nabla V_{\kappa\vv})(x)|\leq 1 \quad \mbox{and} \quad
\|(\nabla^2V_{\kappa\vv})(x)\|\leq
2n\Big(1+\frac{1}{\ln\kappa}\Big)\frac{1}{|x|}
{\bf{1}}_{[\vv/\kappa,\vv]}(|x|) ,
\end{equation}
where $\bf{1}_A(\cdot)$ is the indicator function of the subset
$A\subset\R_+$.

For notation simplicity, set
\begin{equation}\label{b4}
Z(t):=X(t)-Y(t)~~\mbox{ and }~~\LL(t):=
Z(t)-G(X(t-\tau))+G(\bar{Y}(t-\tau)).
\end{equation}
In the sequel, let $t\in[0,T]$ be arbitrary and fix $p\geq2.$ Due to
$\LL(0)={\bf0}\in\R^n$  and $V_{\kappa\vv}({\bf0})=0$, an
application of It\^o's formula gives
\begin{equation*}
\begin{split}
V_{\kappa\vv}(\LL(t)) &= \int_{0}^{t}\langle (\nabla
V_{\kappa\vv})(\LL(s)),\GG_1(s)\rangle \d
s+\frac{1}{2}\int_{0}^{t}\mbox{trace}\{(\GG_2(s))^*(\nabla^2
V_{\kappa\vv})(\LL(s))\GG_2(s)\}\d s\\
&\quad+\int_{0}^{t}\langle \nabla(V_{\kappa\vv})(\LL(s)),\GG_2(s)\d B(s)\rangle \\
& =: I_{1}(t)+I_{2}(t)+I_3(t),
\end{split}
\end{equation*}
where
\begin{equation}\label{b8}
\Gamma_1(t):=b(X(t),X(t-\tau))-b(\bar{Y}(t),\bar{Y}(t-\tau))
\end{equation}
and
\begin{equation*}
\Gamma_2(t):=\si(X(t),X(t-\tau))-\si(\bar{Y}(t),\bar{Y}(t-\tau)).
\end{equation*}
Set
\begin{equation}\label{b9}
V(x,y):=1+|x|^q+| y|^q,x,y\in\R^n.
\end{equation}
According to \eqref{00}, for any $q\ge2$ there exists a constant
$C_T>0$ such that
\begin{equation}\label{b1}
\E \Big(\sup_{0\le t\le T}V(X(t-\tau),\bar{Y}(t-\tau))^q\Big)\le
C_T.
\end{equation}
Noting that
\begin{equation}\label{bao1}
X(t)-\bar{Y}(t) =\LL(t)+\Gamma(t)+ G(X(t-\tau))-G(\bar{Y}(t-\tau)),
\end{equation}
and using  H\"older's inequality and B-D-G's inequality, we get from
\eqref{FunctionBound} and ({\bf A1})-({\bf A2}) that
\begin{equation}\label{b6}
\begin{split}
\Theta(t):&=\E\Big(\sup_{0\leq s\leq
t}|I_{1}(s)|^{p}\Big)+\E\Big(\sup_{0\le
s\le t}|I_3(s)|^p\Big)\\
&\1 \int_{0}^t\{\E |\GG_1(s)|^{p}
+\E\|\GG_2(s)\|^p\}\d s\\
&\1\int_{0}^t\E|X(s)-\bar{Y}(s)|^{p}\d s+\int_{-\tau}^{t-\tau}\E
(V(X(s),\bar{Y}(s))^p|X(s)
-\bar{Y}(s)|^{p})\d s\\
&\1\int_{0}^t\E\{|\LL(s)|^{p}+|\Gamma(s)|^{p}\}\d s+
+\int_{-\tau}^{t-\tau}\E (V(X(s),\bar{Y}(s))^p|X(s)
-\bar{Y}(s)|^{p})\d s.
\end{split}
\end{equation}
Also, by H\"older's inequality, it follows from \eqref{eq4},
 ({\bf A3}) and \eqref{b1}  that
\begin{equation}\label{b2}
\begin{split}
\Theta(t)&\1\int_{0}^t\{\E|\LL(s)|^{p}+\E|\Gamma(s)|^{p}+(\E
V(X(s-\tau),\bar{Y}(s-\tau))^{2p})^{1/2}\\
&\quad\times((\E|Z(s-\tau)|^{2p})^{1/2}+(\E|\GG(s-\tau)|^{2p})^{1/2}\}\d
s\\
&\1\int_{0}^t\{\E|\LL(s)|^{p}+\E|\Gamma(s)|^{p}+(\E|Z(s-\tau)|^{2p})^{1/2}+(\E|\GG(s-\tau)|^{2p})^{1/2}\}\d
s\\
&\1\int_{0}^t\{\E|\LL(s)|^{p}+(\E|Z(s-\tau)|^{2p})^{1/2}+h^{p/2}\}\d
s.
\end{split}
\end{equation}
In the light of \eqref{FunctionBound}-\eqref{bao1}, we derive from
({\bf A1}) that
\begin{equation}\label{I2}
\begin{split}
&\E\Big(\sup_{0\leq s\leq t} |I_{2}(s)|^{p}\Big)\\
&\1\E\int_{0}^t\|(\nabla^2    V_{\kappa\vv})(\LL(s))\|^p\|\GG_2(s)\|^{2p}\d s\\
&\1\E\int_{0}^{t}\frac{1}{|\LL(s)|^{p}}\{|X(s)-\bar{Y}(s)|^{2p}+V(X(s-\tau),\bar{Y}(s-\tau))^{2p}\\
&\quad\times(|X(s-\tau)-\bar{Y}(s-\tau)|^{2p})\}\mbox{\bf I}_{[\vv/
\kappa,\vv]}(|\LL(s)|)\d s\\
&\1\E\int_{0}^t\frac{1}{|\LL(s)|^{p}}\{|\LL(s)|^{2p}+|\GG(s)|^{2p}+|G(X(s-\tau))-G(\bar{Y}(s-\tau))|^{2p}\\
&\quad+V(X(s-\tau),\bar{Y}(s-\tau))^{2p}(|X(s-\tau)-\bar{Y}(s-\tau)|^{2p})\}{\bf{I}_{[\vv/\kappa,\vv]}}(|\LL(s)|)\d s\\
&\1\E\int_{0}^t\frac{1}{|\LL(s)|^{p}}\{|\LL(s)|^{2p}+|\GG(s)|^{2p}\\
&\quad+V(X(s-\tau),\bar{Y}(s-\tau))^{2p}(|X(s-\tau)-\bar{Y}(s-\tau)|^{2p})\}{\bf{I}_{[\vv/\kappa,\vv]}}(|\LL(s)|)\d s\\
&\1\E\int_{0}^{t}\{|\LL(s)|^{p}+\vv^{-p}|\Gamma(s)|^{2p}\\
&\quad+\vv^{-p}V^{2p}(X(s-\tau),\bar{Y}(s-\tau))(|X(s-\tau)-\bar{Y}(s-\tau)|^{2p})\}\d s\\
&\1\int_{0}^t\{\vv^{-p}h^{p}+\E|\LL(s)|^{p}+\vv^{-p}(\E(|Z(s-\tau)|^{4p}))^{1/2}\}\d
s,
\end{split}
\end{equation}
where in the last step we have used H\"older's inequality. Now,
according to  \eqref{FunctionProperty}, \eqref{b2} and \eqref{I2},
one has
\begin{equation*}
\begin{split}
\E\Big(\sup_{0\leq s\leq t} |\LL(s)|^{p}\Big)&\1 \epsilon^{p}+\E\Big(\sup_{0\leq s\leq t}V_{\kappa\vv}(\LL(s))\Big)\\
&\1 \epsilon^{p}+\Theta(t)+\E\Big(\sup_{0\leq s\leq t}
|I_{3}(s)|^{p}\Big)\\
&\1\vv^p+\int_{0}^t\{h^{p/2}+\vv^{-p}h^{p}+\E\Big(\sup_{0\leq r\leq
s}|\LL(r)|^p\Big)\\
&\quad+(\E(|Z(s-\tau)|^{2p}))^{1/2}+\vv^{-p}(\E(|Z(s-\tau)|^{4p}))^{1/2}\}\d
s.
\end{split}
\end{equation*}
Thus, Gronwall's inequality  gives that
\begin{equation}\label{eq2}
\begin{split}
\E\Big(\sup_{0\le s\le t}|\LL(s)|^p\Big)
&\1\vv^p+h^{p/2}+\vv^{-p}h^{p}\\
&\quad+\int_0^{(t-\tau)\vee0}\{(\E(|Z(s)|^{2p}))^{1/2}+\vv^{-p}(\E(|Z(s)|^{4p}))^{1/2}\}\d s\\
&\leq h^{p/2}+\int_0^{(t-\tau)\vee0}\{(\E(|Z(s)|^{2p}))^{1/2}+\vv^{-p}(\E(|Z(s)|^{4p}))^{1/2}\}\d s,\\
\end{split}
\end{equation}
by choosing $\vv=h^{1/2}$ and taking $|Z(t)|\equiv0$ for
$t\in[-\tau,0]$ into account. Next,  by  ({\bf A1}) and \eqref{b1}
it follows from H\"older's inequality that
\begin{equation}\label{eq3}
\begin{split}
&\E\Big(\sup_{0\leq t\leq T}|Z(t)|^{p}\Big)\1 \E\Big(\sup_{0\leq
t\leq T} |\LL(t)|^{p}\Big)+\E\Big(\sup_{-\tau\leq t\leq
T-\tau}|G(X(t))
-G(\bar{Y}(t))|^{p}\Big)\\
&\1 \E\Big(\sup_{0\leq t\leq T}
|\LL(t)|^{p}\Big)+\E\Big(\sup_{-\tau\leq t\leq
T-\tau}(V(X(t),\bar{Y}(t))^{p}|X(t)-
\bar{Y}(t)|^{p})\Big)\\
&\1 \E\Big(\sup_{0\leq t\leq T} |\LL(t)|^{p}\Big)+h^{p/2}+ \Big(\E\Big(\sup_{0\leq t\leq (T-\tau)\vee0}|Z(t)|^{2p}\Big)\Big)^{1/2}.\\
\end{split}
\end{equation}
Substituting \eqref{eq2} into \eqref{eq3} yields that
\begin{equation}\label{eq6}
\begin{split}
\E\Big(\sup_{0\le t\le T}|Z(t)|^p\Big)& \1h^{p/2}+ \Big(\E\Big(\sup_{0\leq t\leq (T-\tau)\vee0}|Z(t)|^{2p}\Big)\Big)^{1/2}\\
&+\int_{0}^{(T-\tau)\vee0}\{(\E(|Z(t)|^{2p}))^{1/2}+\vv^{-p}(\E(|Z(t)|^{4p}))^{1/2}\}\d
t.
\end{split}
\end{equation}
 Hence, we have
\begin{equation*}
\E\Big(\sup_{0\le t\le \tau}|Z(t)|^p\Big) \1 h^{p/2},
\end{equation*}
and
\begin{equation*}
\begin{split}
\E\bigg(\sup_{0\le t\le 2\tau}|Z(t)|^p\bigg)&\1h^{p/2}+ \Big(\E\Big(\sup_{0\leq t\leq \tau}|Z(t)|^{2p}\Big)\Big)^{1/2}\\
&+\int_{0}^\tau\{(\E(|Z(t)|^{2p}))^{1/2}+\vv^{-p}(\E(|Z(t)|^{4p}))^{1/2}\}\d
t\\
& \1 h^{p/2}
\end{split}
\end{equation*}
by taking $\vv=h^{1/2}$. Thus, the desired assertion (\ref{0})
follows from an inductive argument.

\section{Proof of Theorem \ref{ConvergenceRateJump}}
Hereinafter, $(X(t))$ is the strong solution to \eqref{NSDDEJump}
and $(Y(t))$ is the  continuous-time EM scheme (i.e.,
\eqref{CEulerschemeJump}) associated with \eqref{NSDDEJump}.

The lemma below plays a crucial role in revealing convergence rate
of the EM scheme.
\begin{lem}\label{L2}
{\rm Under ({\bf A1})-({\bf A5}) with $\si\equiv{\bf0}_{n\times m}$
therein,
\begin{equation}\label{JumpBound}
\E\Big(\sup_{0\leq t\leq T}|X(t)|^{p}\Big) \vee \E\Big(\sup_{0\leq
t\leq T}|Y(t)|^{p}\Big) \leq C_T,~~~p\ge2
\end{equation}
for some constant $C_T>0$ and
\begin{equation}\label{eq14}
\E\Big(\sup_{0\le t\le T}|\GG(t)|^{p}\Big)\1h,~~~p\ge2,
\end{equation}
where $\GG(t):=Y(t)-\bar{Y}(t)$.}
\end{lem}
\begin{proof}
By carrying out a similar argument to that of \cite[Theorem 3.1,
p.210]{Mao}, \eqref{NSDDEJump} admits a unique strong solution
$\{X(t)\}$ according to \cite[Theorem 117, p.79]{Situ}. On the other
hand, since the proof of \eqref{L2}  is quite similar to that of
\eqref{00}, we here omit its detailed proof.

\smallskip

In the sequel, we aim to show \eqref{eq14}.  From ({\bf A4}), it
follows that
\begin{equation}\label{b3}
|g(x,y,u)|\leq C(1+|x|+|y|^{1+q})|u|^r,~~~x,y\in\R^n,~u\in U
\end{equation}
for some $C>0.$ Applying B-D-G's inequality (see, e.g., Lemma
\cite[Theorem 1]{{MAR}}) and H\"older's inequality, we derive   that
\begin{equation*}
\begin{split}
\E\Big(\sup_{0\le t\le T}|\GG(t)|^{p}\Big)
&\1\sup_{0\le k\le M-1}\Big\{\E\Big(\sup_{kh\le t\le(k+1)h}\Big|\int_{kh}^tb(\bar{Y}(s),\bar{Y}(s-\tau))\d s\Big|^p\Big)\\
&\quad+\E\Big(\sup_{kh\le t\le(k+1)h}\Big|\int_{kh}^t\int_{U}g(\bar{Y}(s-),\bar{Y}((s-\tau)-),u) \tilde{N}(\d s,\d u)\Big|^p\Big)\Big\}\\
& \1 \sup_{0\le k\le M-1}\Big\{\int_{kh}^{(k+1)h}\Big(h^{p-1}\E|b(\bar{Y}(s),\bar{Y}(s-\tau))|^{p}\\
&\quad+\int_{U}\E|g(\bar{Y}(s),\bar{Y}(s-\tau),u)|^{p}\lambda(\d u)\Big)\d s\Big\}\\
& \1 \sup_{0\le k\le M-1}\Big\{\int_{kh}^{(k+1)h}\Big(1+\E\Big(\sup_{-\tau\le r\le s}|Y(r)|^{p(1+q)}\Big)\Big)\\
& \quad\times\Big(h^{p-1}+\int_{U}|u|^{pr}\lambda(\d u)\Big)\d s\Big\}\\
&\1 h^{p}+h\\
&\1 h,
\end{split}
\end{equation*}
where we have used {\bf(A2)} with $\si\equiv{\bf0}_{n\times m}$ and
\eqref{b3} in the third step,    and  \eqref{JumpBound} and
{\bf(A5)} in the last two step, respectively. So \eqref{eq14}
follows as required.
\end{proof}

Next, we go back to finish the proof of Theorem
\ref{ConvergenceRateJump}.

 \noindent {\bf Proof of Theorem \ref{ConvergenceRateJump}.} We follow the
 idea of the proof for Theorem \ref{ConvergenceRate} to complete the proof.
Set
\begin{equation*}
\Gamma_3(t,u):= g(X(t),X(t-\tau),u)-g(\bar{Y}(t),\bar{Y}(t-\tau),u).
\end{equation*}
Applying It\^o's formula as well as the Lagrange mean value theorem
to $V_{\kappa\vv}(\cdot)$, defined by \eqref{b5},  gives that
\begin{equation*}
\begin{split}
 V_{\kappa\vv}(\LL(t)) &=
\int_{0}^{t}\langle (\nabla V_{\kappa\vv})
(\LL(s)),\Gamma_{1}(s)\rangle \d s\\
&
\quad+\int_{0}^{t}\int_{U}\{V_{\kappa\vv}(\LL(s)+\Gamma_3(s))-V_{\kappa\vv}(\LL(s))-\langle
(\nabla V_{\kappa\vv})
(\LL(s)),\Gamma_3(s)\rangle\}\lambda(\d u)\d s\\
&\quad+\int_{0}^{t}\int_{U}\{V_{\kappa\vv}(\LL(s-)+\Gamma_3(s-))-V_{\kappa\vv}(\LL(s-))\}\tilde{N}(\d u,\d s)\\
&=V_{\kappa\vv}(\LL(0))+\int_{0}^{t}\langle (\nabla V_{\kappa\vv})
(\LL(s)),\Gamma_{1}(s)\rangle \d s\\
& \quad+\int_{0}^{t}\int_{U}\Big\{\int_{0}^{1}\langle \nabla
V_{\kappa\vv}(\LL(s)+r\Gamma_3(s))
- \nabla V_{\kappa\vv}(\LL(s)),\Gamma_3(s)\rangle \d r\Big\}\lambda(\d u)\d s\\
&\quad+\int_{0}^{t}\int_{U}\Big\{\int_{0}^{1}\langle \nabla
V_{\kappa\vv}(\LL(s-)+r\Gamma_3(s-)),\Gamma_3(s-)\rangle
\d r\Big\}\tilde{N}(\d u,\d s)\\
&\quad =:J_{1}(t)+J_{2}(t)+J_3(t),
\end{split}
\end{equation*}
in which $\GG_1$ is defined as in \eqref{b8}. By  B-D-G's inequality
(see, e.g., Lemma \cite[Theorem 1]{{MAR}}), we obtain from
\eqref{FunctionBound}, \eqref{b6} with $\si\equiv{\bf0}_{n\times m}$
 therein, ({\bf A4}) and ({\bf A5}) that
\begin{equation*}
\begin{split}
\Upsilon(t):=\sum_{i=1}^3\E\Big(\sup_{0\leq s\leq
t}|J_i(s)|^{p}\Big)
&\1\int_{0}^t\E\{|\LL(s)|^{p}+|\Gamma(s)|^{p}\}\d s\\
&\quad+\int_{-\tau}^{t-\tau}\E (V(X(s),\bar{Y}(s))^p|X(s)
-\bar{Y}(s)|^{p})\d s,
\end{split}
\end{equation*}
where $V(\cdot,\cdot)$ is introduced in \eqref{b9}. Observe from
H\"older's inequality that
\begin{equation}\label{eq35}
\begin{split}
&\E V(X(s),\bar{Y}(s))^{p}|Y(s)-\bar{Y}(s)|^{p}\\
&\1 (\E V(X(s),\bar{Y}(s))^{\frac{p(1+\theta)}{\theta}})^{\frac{\theta}{1+\theta}}(\E |X(s)-\bar{Y}(s)|^{p(1+\theta)})^{\frac{1}{1+\theta}}\\
&\1(\E V(X(s),\bar{Y}(s))^{\frac{p(1+\theta)}{\theta}})^{\frac{\theta}{1+\theta}} (\E|Z(s)|^{p(1+\theta)}+\E|\Gamma(s)|^{p(1+\theta)})^{\frac{1}{1+\theta}}\\
&\1(\E(1+|X(s)|^{\frac{pq(1+\theta)}{\theta}}+|\bar{Y}(s)|^{\frac{pq(1+\theta)}{\theta}})^{\frac{\theta}{1+\theta}}(\E|Z(s)|^{p(1+\theta)}
+\E|\Gamma(s)|^{p(1+\theta)})^{\frac{1}{1+\theta}}\\
&\1 ( \E|Z(s)|^{p(1+\theta)})^{\frac{1}{1+\theta}}+( \E|\Gamma(s)|^{p(1+\theta)})^{\frac{1}{1+\theta}}\\
&\1 h^{\frac{1}{1+\theta}}+(
\E|Z(s)|^{p(1+\theta)})^{\frac{1}{1+\theta}},~~~\theta>0,
\end{split}
\end{equation}
in which we have used \eqref{JumpBound} in the penultimate display
and \eqref{eq14} in the last display, respectively. So we arrive at
\begin{equation*}
\begin{split}
\Upsilon(t)
&\1h^{\frac{1}{1+\theta}}+\int_{0}^t\E\{|\LL(s)|^{p}+|\Gamma(s)|^{p}\}\d
s+\int_{-\tau}^{t-\tau}(
\E|Z(s)|^{p(1+\theta)})^{\frac{1}{1+\theta}}\d s.
\end{split}
\end{equation*}
This,  together with \eqref{FunctionProperty} and \eqref{eq14},
implies
\begin{equation*}
\begin{split}
\E\Big(\sup_{0\leq s\leq t} |\LL(t)|^{p}\Big)&\1 \epsilon^{p}+\E\Big(\sup_{0\leq s\leq t}V_{\kappa\vv}(\LL(s))\Big)\\
&\1\epsilon^{p}+
h^{\frac{1}{1+\theta}}+\int_{0}^t\E\{|\LL(s)|^{p}+|\Gamma(s)|^{p}\}\d
s+\int_{-\tau}^{t-\tau}(
\E|Z(s)|^{p(1+\theta)})^{\frac{1}{1+\theta}}\d s\\
&\1 h^{\frac{1}{1+\theta}}+\int_{0}^t\E|\LL(s)|^{p}\d
s+\int_{-\tau}^{t-\tau}(
\E|Z(s)|^{p(1+\theta)})^{\frac{1}{1+\theta}}\d s
\end{split}
\end{equation*}
by taking $\vv=h^{\frac{1}{p(1+\theta)}}$ in the last display. Using
Gronwall's inequality, due to $Z(\theta)=0$ for
$\theta\in[-\tau,0]$, one has
\begin{equation*}
\begin{split}
\E\Big(\sup_{0\leq t\leq T}
|\LL(t)|^{p}\Big)&\1h^{\frac{1}{1+\theta}}+\int_0^{(T-\tau)\vee0}(
\E|Z(s)|^{p(1+\theta)})^{\frac{1}{1+\theta}}\d s.
\end{split}
\end{equation*}
Next, observe from ({\bf A1}) and H\"older's inequality that
\begin{equation}\label{b7}
\begin{split}
&\E\Big(\sup_{0\leq t\leq T}|Z(t)|^{p}\Big)\1 \E\Big(\sup_{0\leq
t\leq T} |\LL(t)|^{p}\Big)+\E\Big(\sup_{-\tau\leq t\leq
T-\tau}|G(X(t))
-G(\bar{Y}(t))|^{p}\Big)\\
&\1 \E\Big(\sup_{0\leq t\leq T}
|\LL(t)|^{p}\Big)+\E\Big(\sup_{-\tau\leq t\leq
T-\tau}(V(X(t),\bar{Y}(t))^{p}|X(t)-
\bar{Y}(t)|^{p})\Big)\\
&\1\E\Big(\sup_{0\leq t\leq T} |\LL(t)|^{p}\Big)\\
&\quad+\Big\{\Big(1+\E\Big(\sup_{-\tau\leq t\leq
T}|X(t)|^{\frac{pq(1+\theta)}{\theta}}\Big)+\Big(\sup_{-\tau\leq
t\leq
T}|Y(t)|^{\frac{pq(1+\theta)}{\theta}}\Big)\Big\}^{\frac{\theta}{1+\theta}}\\
&\quad\times\Big\{\E\Big(\sup_{-\tau\leq t\leq
T-\tau}|Z(t)|^{p(1+\theta)}\Big) +\E\Big(\sup_{-\tau\leq t\leq
T}|\Gamma(t)|^{p(1+\theta)}\Big)\Big\}^{\frac{1}{1+\theta}}\\
 &\1h^{\frac{1}{1+\theta}}+\Big(\E\Big(\sup_{0\leq t\leq
(T-\tau)\vee0}|Z(t)|^{p(1+\theta)}\Big)\Big)^{\frac{1}{1+\theta}},
\end{split}
\end{equation}
where in the last step we have utilized \eqref{JumpBound} and
\eqref{eq14}. So we find that
\begin{equation*}
\E\Big(\sup_{0\leq  t\leq  \tau}|Z(t)|^{p}\Big)\1
h^{\frac{1}{1+\theta}},
\end{equation*}
which, in addition to \eqref{b7}, further yields that
\begin{equation*}
\begin{split}
\E\Big(\sup_{0\leq t\leq 2\tau}|Z(t)|^{p}\Big)&\1
h^{\frac{1}{1+\theta}}+\Big(\E\Big(\sup_{0\leq t\leq
\tau}|Z(t)|^{p(1+\theta)}\Big)\Big)^{\frac{1}{1+\theta}}\\
&\1 h^{\frac{1}{(1+\theta)^{2}}}+h^{\frac{1}{1+\theta}}\\
&\1 h^{\frac{1}{(1+\theta)^{2}}}.
\end{split}
\end{equation*}
Thus, the desired assertion follows from an inductive argument.


\begin{thebibliography}{99}
{\small

\setlength{\baselineskip}{0.14in}
\parskip=0pt

\bibitem{AAM} Appleby, John A.~D., Appleby-Wu, H., Mao, X., On the almost sure running maximum of solutions of affine neutral
stochastic functional differential equations, arXiv:1310.2349v1.







\bibitem{BWY} Bao, J., Wang, F.-Y., Yuan, C., Transportation cost inequalities for neutral functional stochastic equations,
 {\it Z. Anal. Anwend.}, {\bf32} (2013),   457--475.

\bibitem{Bao} Bao,J.,  Yuan, C.,  Convergence rate of EM scheme for SDDEs. {\sl
Proc. Amer. Math. Soc., } {\bf 141}(2013),  3231--3243.

\bibitem{Bao11} Bao, J., B\"ottcher, B., Mao, X., Yuan, C., Convergence rate of numerical solutions to SFDEs with jumps,
{\it J. Comput. Appl. Math.}, {\bf236} (2011),   119--131.


\bibitem{BHY} Bao, J., Hou, Z., Yuan, C., Stability in distribution of neutral stochastic differential delay equations with Markovian
switching,
  {\it Statist. Probab. Lett.}, {\bf79} (2009),  1663--1673.



\bibitem{BY}
Bao, J., Yuan, C.,    Large deviations for neutral functional
SDEs with jumps, {\sl Stochastics},
{\bf 87}(2015), 48-70.


\bibitem{GS13}Gy\"ongy, I.,  Sabanis,  S., A Note on Euler Approximations for
Stochastic Differential Equations with Delay, {\it Appl. Math.
Optim.}, {\bf68} (2013),  391--412.




\bibitem{Gyo}
Gy\"ongy, I., R\'asonyi, M.,  A note on Euler approximations for SDEs with H\"older continuous diffusion coefficients. {\sl Stochastic Process. Appl.,} {\bf 121}(2011),2189-2200.





\bibitem{Jiang}
Jiang, F., Shen, Y., and Wu, F.,  A note on order of convergence of numerical method for neutral stochastic functional differential equations.
{\sl Commun. Nonlinear Sci. Numer. Simul.,} {\bf 17} (2012) 1194-1200.


\bibitem{LC}Li,X., Cao, W., On mean-square stability of two-step Maruyama methods for nonlinear
neutral stochastic delay differential equations, {\it Appl. Math.
Comput.},  {\bf261} (2015), 373--381.


\bibitem{Mao16} Mao, X., Convergence rates of the truncated Euler-Maruyama method for stochastic differential equations,
 {\it J. Comput. Appl. Math.}, {\bf296} (2016), 362--275.

\bibitem{Mao}
Mao, X.,  \emph{Stochastic differential equations and applications.} Second Edition. Horwood Publishing Limited, 2008. Chichester.


\bibitem{MSY} Mao, X., Shen, Y., Yuan, C., Almost surely asymptotic stability of neutral stochastic differential delay equations with Markovian
switching,
  {\it Stochastic Process. Appl.}, {\bf118} (2008),   1385--1406.


\bibitem{MAR}
 Marinelli, C.,   R\"ockner,  M.,   On Maximal inequalities for purely discountinuous martingales in infinite dimensional, S\`eminnaire de Probabilit\`es XLVI, Lecture Notes in Mathematics 2123 (2014),  293-316.


\bibitem{M11} Milosevic, M., Highly nonlinear neutral stochastic differential equations with
time-dependent delay and the Euler-Maruyama method, {\it Math.
Comput. Modelling}, {\bf54} (2011),   2235--2251.


\bibitem{S13}Sabanis, S., A note on tamed Euler approximations, {\it Electron. Commun. Probab.}, {\bf18} (2013),   1--10.


\bibitem{Situ}
Situ, R., \emph{Theory of stochastic differential equations with jumps and applications.}  Mathematical and Analytical Techniques with Applications to Engineering. Springer, 2005.  New York.

\bibitem{Wu}
Wu, F.,  Mao, X.,   Numerical solutions of neutral stochastic functional differential equations. {\sl SIAM J. Numer. Anal.,} {\bf 46} (2008), 1821-1841.


\bibitem{Yu} Yu, Z., Almost sure and mean square exponential stability of numerical
solutions for neutral stochastic functional differential equations,
{\it Int. J. Comput. Math.}, {\bf 92} (2015),  132--150.



\bibitem{YM} Yuan, C.,  Mao, X.,  A note on the rate of convergence of the Euler-Maruyama method for stochastic differential
equations, {\it Stoch. Anal. Appl.}, {\bf26} (2008),  325--333.

\bibitem{YG} Yuan, C., Glover, W., Approximate solutions of stochastic differential delay equations with Markovian
switching,
 {\it J. Comput. Appl. Math.}, {\bf194} (2006),   207--226.

\bibitem{Zhou1}  Zhou, S.,  Fang, Z.,  Numerical approximation of nonlinear neutral stochastic functional
differential equations, {\it J. Appl. Math. Comput.}, {\bf41}
(2013),
 427--445.



\bibitem{Zhou}
Zong, X., Wu, F. and Huang, C.,  Exponential mean square stability of the theta approximations for neutral stochastic differential delay equations. {\sl J. Comput. Appl. Math.,} {\bf 286} (2015), 172-€"185.

\bibitem{Zong}Zong, X., Wu, F., Exponential stability of the exact and
numerical solutions for neutral stochastic delay differential
equations,  {\it Appl. Math. Model.}, {\bf40} (2016),   19--30.


}

\end{thebibliography}
\end{document}